\documentclass{article}

\usepackage{amsfonts,amsmath,amssymb}
\usepackage{graphicx}
\usepackage{float}
\usepackage{authblk}
\newtheorem{satz}{Theorem}

\newtheorem{theorem}[satz]{Theorem}
\newtheorem{lemma}[satz]{Lemma}
\newtheorem{definition}[satz]{Definition}

\newtheorem{problem}[satz]{Problem}

\newtheorem{conjecture}[satz]{Conjecture}
\newtheorem{claim}[satz]{Claim}

\newcommand{\qed}{{} \hfill \mbox{$\Box$}}

\def\({\big (}
\def\){\big )}

\def\_phi{\varphi}

\title{On the structure of pointsets with many collinear triples}

\author
{
J\'{o}zsef Solymosi
\thanks{Department of Mathematics,
The University of British Columbia, 
Vancouver, Canada and Obuda University, Budapest, Hungary.
Email: solymosi@math.ubc.ca}}
\date{} % Activate to display a given date or no date
\begin{document}
\maketitle

\begin{abstract}
    It is conjectured that if a finite set of points in the plane contains many collinear triples, then part of the set has a structure. We will show that under some combinatorial conditions, such pointsets have special configurations of triples, proving a case of Elekes' conjecture. Using the techniques applied in the proof, we show a density version of Jamison's theorem. If the number of distinct directions between many pairs of points of a point set in a convex position is small, then many points are on a conic.
\end{abstract}

\section{Introduction}
In this paper, we investigate planar pointsets spanning many collinear triples. If three $n$-element pointsets are well separated (mutually avoiding) and there are $cn^2$ collinear triples, one point from each, it contains a Pappus configuration. Elekes conjectured it for general pointsets containing many collinear triples. We prove a special case in Theorem \ref{main1}. If one of the three separated pointsets consists of collinear points, then a strong structural result is proved in Lemma \ref{main3}: a positive fraction of the remaining points are on a conic. As a corollary, we prove another conjecture of Elekes in Theorem \ref{cor:elekes}. If the number of distinct directions between many pairs of points of a point set in a convex position is small, then many points are on a conic.

It is a classical and essential problem in discrete geometry to understand the structure of finite planar pointsets defining many lines containing at least three points of this set of points.
An excellent survey article on many related problems is written by Borwein and Moser \cite{BM}, and several open problems are mentioned in the problem book of Brass, Moser, and Pach (Chapter 7 in \cite{BMP}).
Solving a famous conjecture of Dirac and Motzkin, the following statement was proved by Green and Tao in 2013 \cite{GT}. 

\begin{theorem} Suppose that $P$ is a set of
$n$ points in the plane. Then there are at most $\lfloor \frac{n(n-3)}{6}\rfloor+1$ lines that contain at least three points of $P,$ provided $n$ is large enough.
\end{theorem}

For a given set of points, $P,$ a line is determined by $P$ if it contains at least two points of $P.$ If such a line has exactly two points, it is called  an {\em ordinary} line.
Green and Tao proved a strong structure theorem which states that if $P$ has at most $Kn$ ordinary lines, then all but $O(K)$ points of $P$ lie on a cubic curve if $n$ is sufficiently large depending on $K$ (in this result $K$ is a constant or a very slowly growing function of $n$). Unfortunately, almost nothing is known about the structure of $P$ when there are $cn^2$ lines with at least three points of $P,$ where $0<c<1/6.$ Elekes conjectured that here, under this weaker condition, ten or more points of $P$ lie on a (possibly
degenerate) cubic, provided that $n\geq n_0(c)$ (see, e.g. Conjecture 2.1 in \cite{ESz}). There are some structural results for $n$-element pointsets with $cn^2$ collinear triples, like in \cite{So1} and \cite{DGOS}, but these are not enough to prove Elekes' conjecture on ``ten points on a cubic''. 
There is another related conjecture (problem) of Elekes (see in \cite{Prob} and in \cite{Vera}), which was the main inspiration for this work.

\begin{problem}[Elekes \cite{Prob}]
For fixed $c > 0,$ suppose that the edges of a graph $G$ with $n$ vertices and $cn^2$ edges are well-coloured using $n$ colours (i.e., no two edges of the same colour are incident upon a common vertex). If $n$ is sufficiently large, must $G$ contain a six-cycle with opposite edges having the same colour? 
\end{problem}

Elekes gave another formulation of the problem in \cite{Prob}. If the colours of the edges are represented by extra vertices, each added to the edges with its colour, then we can state the problem (now stated as a conjecture) in terms of 3-uniform hypergraphs. A 3-uniform hypergraph is {\em linear} if any two edges have at most one vertex common. 
In a 3-uniform hypergraph, six edges are called a {\em tic-tac-toe} if they intersect each other like the rows and columns of a $(3\times 3)$ tic-tac-toe board. 

\begin{figure}[H]
\centering
\includegraphics[scale=.3]{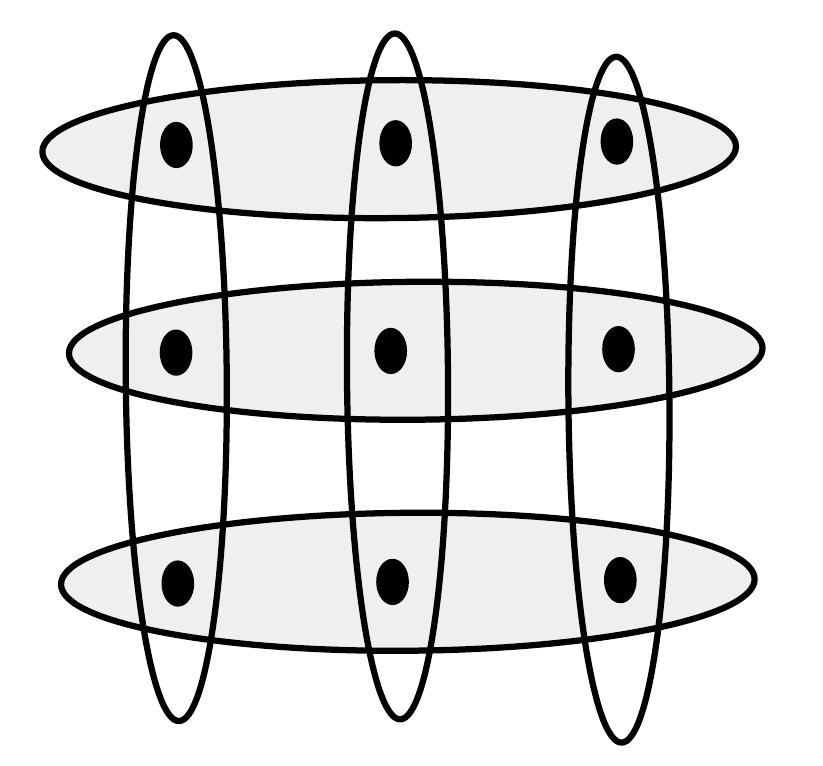}
\caption{A tic-tac-toe hypergraph on nine vertices with six edges}
\label{fig:tictac}
\end{figure}

It follows from the converse to Pascal's theorem, called Braikenridge–Maclaurin theorem \cite{MBM} that in a tic-tac-toe configuration of collinear triples in the plane if three pairwise independent\footnote{Two vertices are independent if there is no edge containing both.} points are collinear, then the remaining six points are on a conic (Figure \ref{fig:Pascal}).
A 3-uniform linear hypergraph on $n$ vertices has at most $\lfloor\frac{n(n-1)}{6}\rfloor$ edges. Elekes conjectured that if it contains no tic-tac-toe, then it should be sparse.

\begin{conjecture}\label{Pascal}
If a 3-uniform linear hypergraph on $n$ vertices contains no tic-tac-toe, it is sparse; it has $o(n^2)$ edges. (For any $\varepsilon>0$ there is an $n_0=n_0(\varepsilon)$ such that if $n\geq n_0$ then the number of edges is at most $\varepsilon n^2.$)
\end{conjecture}

This conjecture was refuted by Gishboliner and Shapira, giving a nice construction for 3-uniform hypergraphs on $n$ vertices with $\sim\frac{n^2}{16}$ edges without tic-tac-toe \cite{GS}.
Earlier, F\"uredi and Ruszink\'o even conjectured that there are 3-uniform hypergraphs on $n$ vertices with $\sim\frac{n^2}{6}$ edges (almost Steiner triple systems) without a tic-tac-toe \cite{FR}. This conjecture is still open.

\begin{figure}[H]
\centering
\includegraphics[width=7cm, angle=-90]{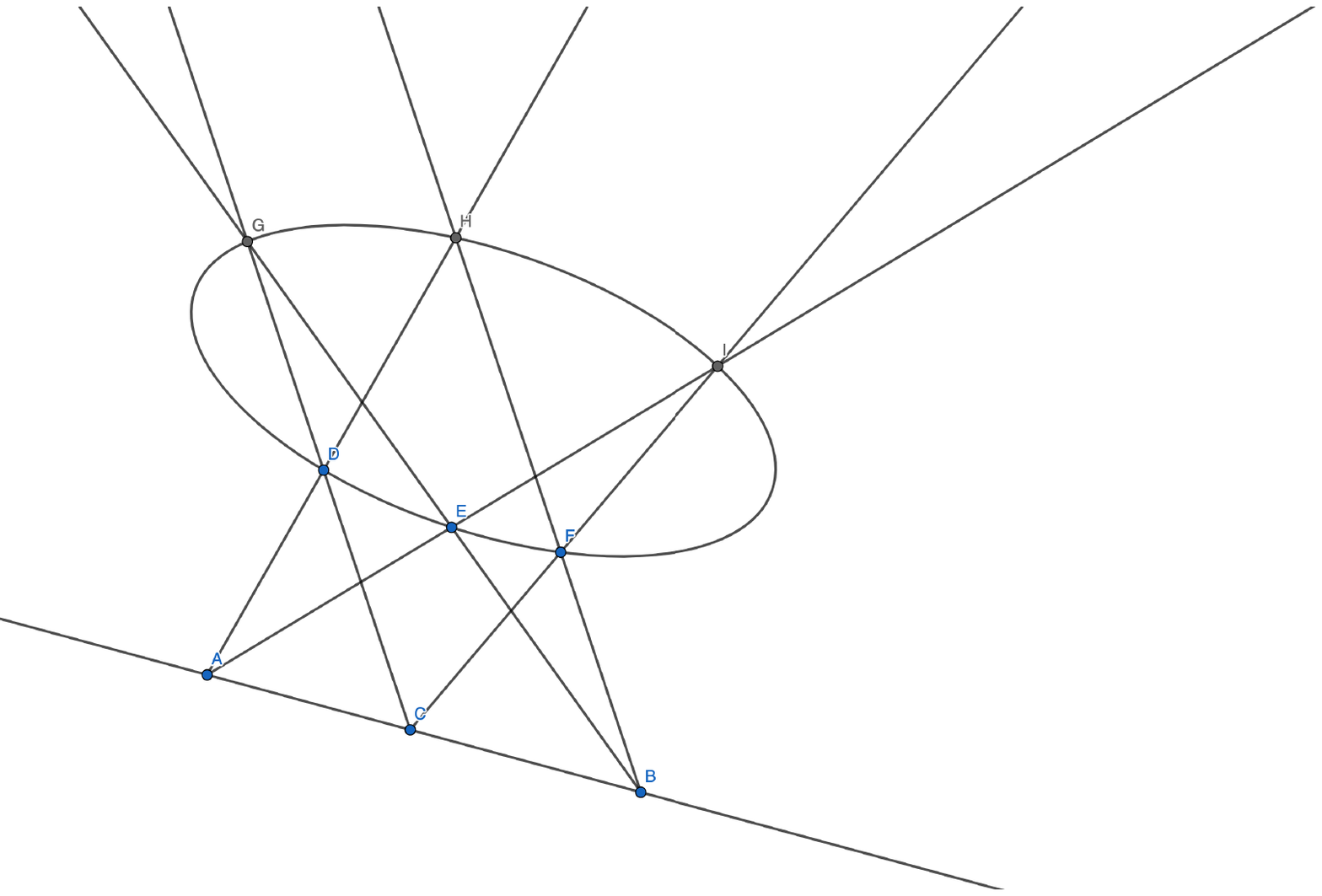}
\caption{If $A, B,C$ are collinear in a tic-tac-toe configuration of collinear triples, then the other six points are on a conic}
\label{fig:Pascal}
\end{figure}

In this paper, we will prove Conjecture \ref{Pascal} when the edges are collinear triples in some pointsets. We show that if the edges of a 3-uniform hypergraph are defined by collinear triples from three sets of points, there are vertices spanning many edges. 
%more than what would follow from the Brown-Erd\H{o}s-S\'{o}s conjecture. 
In the last section, we prove that if one of the three sets consists of collinear points under some combinatorial conditions, then a large fraction of the other two sets are on a conic. As noted in \cite{E_SP} (see the quote after Conjecture \ref{Elekes6} below), without extra conditions, this statement would not hold. For example, taking a projective image of an integer $\lfloor\sqrt{n}\rfloor\times\lfloor\sqrt{n}\rfloor$ square grid with the line of infinity,  there are $cn^2$ collinear triples with two points in the grid and one on the line at infinity (the points of the directions), while the grid contains at most $n^{1/4}$ points of a parabola and $n^{o(1)}$ points of other conics\footnote{This claim can be recovered from multiple sources, see, e.g. Lemma 5. in \cite{KS}.} (Figure \ref{fig:Grid}). 

\begin{figure}[h!]
\centering
\includegraphics[width=7cm]{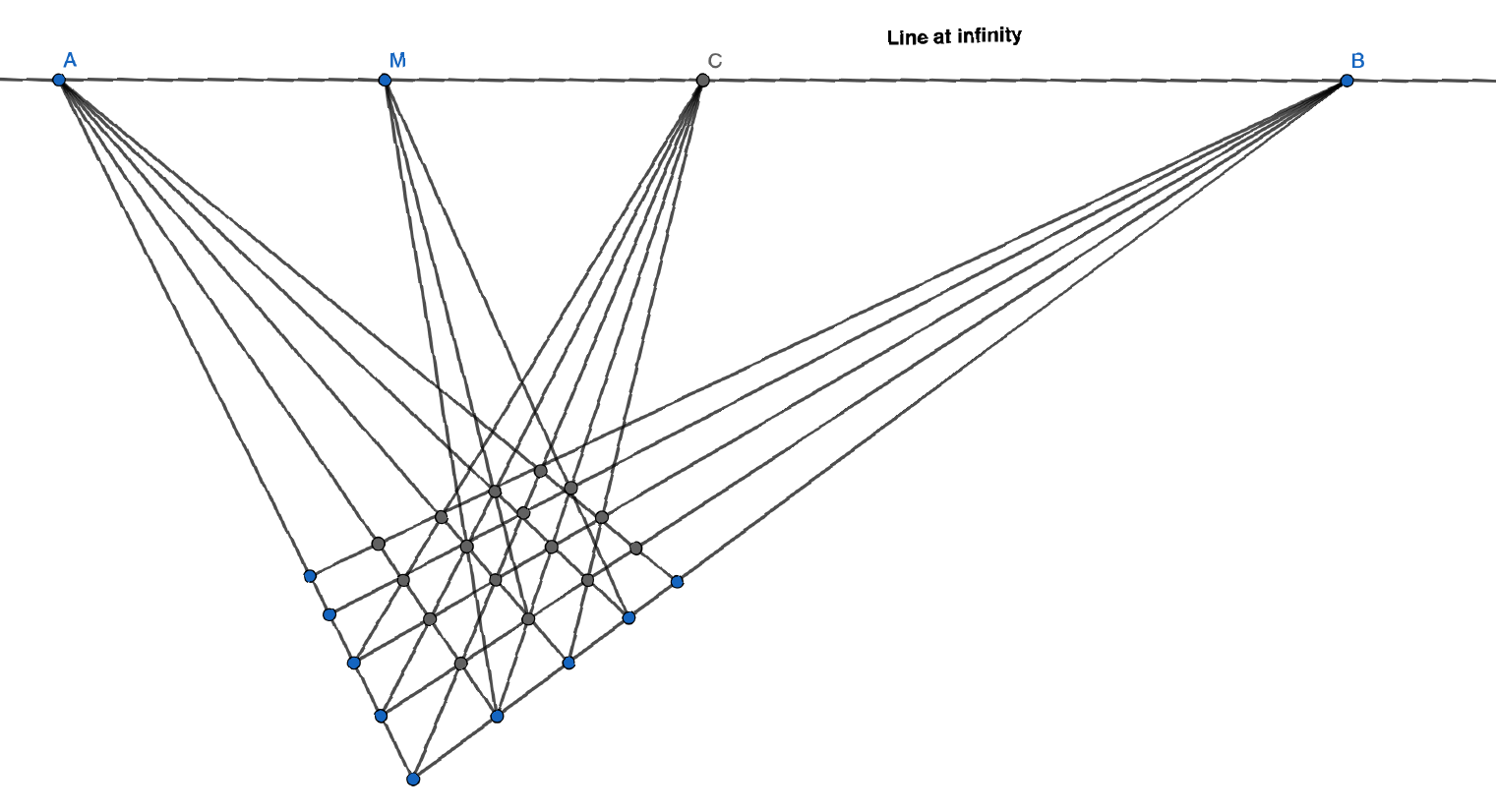}
\caption{$A, B,C,$ and $M$ are some of the points of directions on the line at infinity determined by the $5\times 5$ integer grid}
\label{fig:Grid}
\end{figure}

In Section \ref{directions} of the paper, we prove a strong structural result for some pointsets where many pairs determine few distinct directions. Using the combinatorial model for analyzing sets with few directions due to Eli Goodman and Ricky Pollack \cite{GP}, Jamison proved the following structural result.

\begin{theorem}[Jamison \cite{Jam}]\label{Jam}
If an $n$-element pointset contains no three collinear points and determines $n$ directions, then it is affinely equivalent to the vertices of a regular $n$-gon, i.e. the points are on a conic. 
\end{theorem}

 The Goodman-Pollack method is nicely explained in Chapter 12 of the Proofs from the BOOK by Aigner and Ziegler \cite{AZ} where they proved Ungar's theorem \cite{Un} on the minimum number of distinct directions determined by $n$ non-collinear points in the plane.
 
 In the problem of characterizing pointsets determining few directions, we show that if two $n$-element sets, $A$ and $B$, contain $cn^2$ distinct 
$(a,b)\in A\times B$ pairs determining at most $n$ directions, then -- under some combinatorial assumptions -- a positive fraction of the points of $A\cup B$ is on a conic.

\section{Elekes' conjecture for collinear triples}
While Conjecture \ref{Pascal} was disproved in general, we show that it  holds for certain sets of collinear triples. Given three pointsets in the plane, $P,Q,$ and $S,$ we say that the three sets are mutually avoiding sets if no line determined by two points in a set intersects the two convex hulls of the other two pointsets. The concept of mutually avoiding sets is often used in discrete and computational geometry, like in \cite{AEP}, \cite{V}, \cite{MS} and \cite{PRT}. It was proved in \cite{AEP} that any planar $n$-element pointset in general position contains two mutually avoiding subsets of size at least $\sqrt{n/12}.$

\begin{figure}[h!]
\centering
\includegraphics[width=6cm, angle=-90]{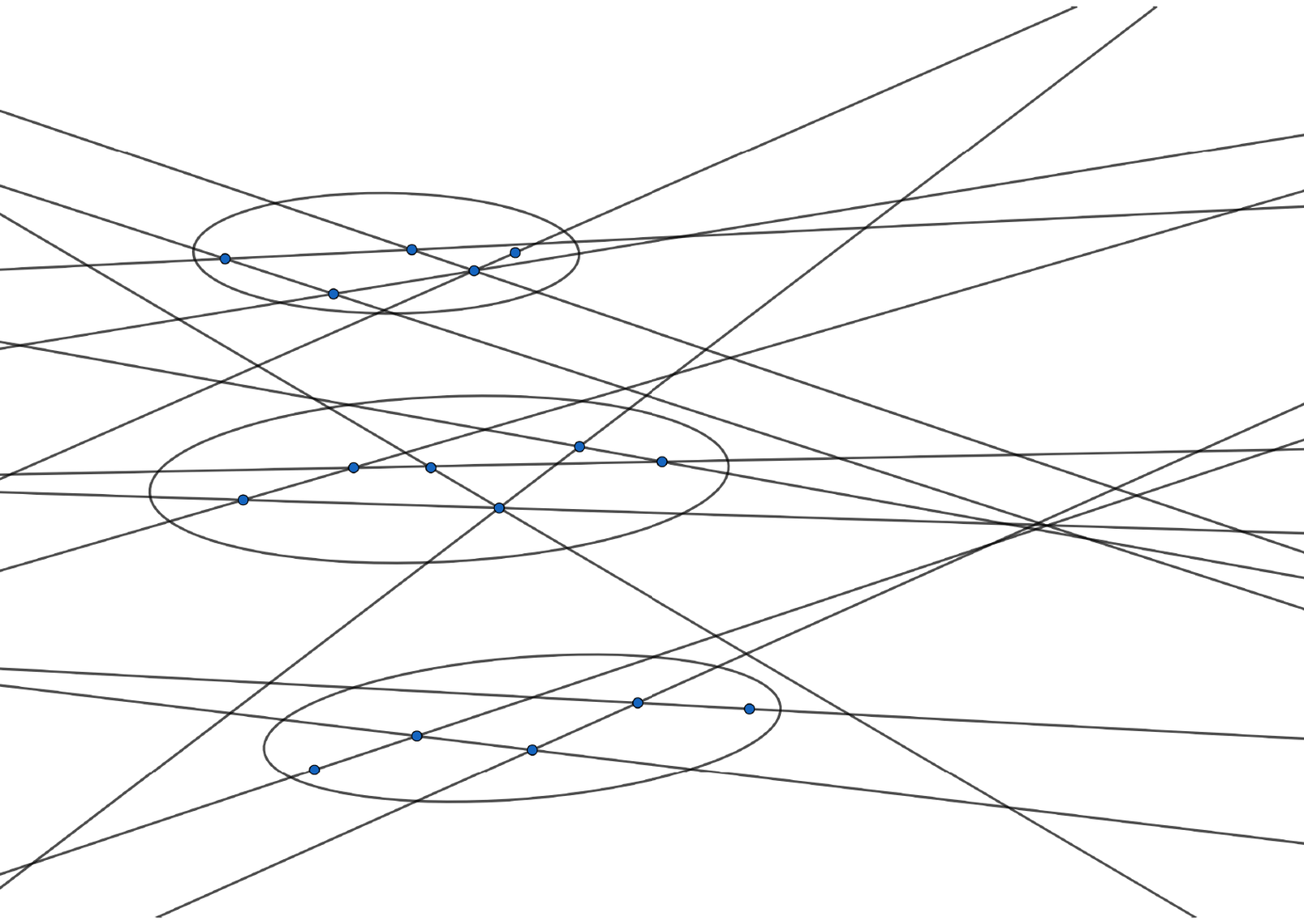}
\caption{The three sets are mutually avoiding}
\label{fig:Mutual}
\end{figure}

\begin{theorem}\label{main1}
For any $\delta>0$ constant, there is a threshold $n_0=n_0(\delta)$, so the following holds. Let $A,B,C$ be mutually avoiding $n$-element pointsets in the plane. Let us suppose that there are at least $\delta n^2$ collinear $(a,b,c)$ triples, where $a\in A, b\in B,$ and $c\in C.$ If $n\geq n_0,$ then there is a tic-tac-toe configuration among the collinear triples.
\end{theorem}

We will show a density version of Theorem \ref{main1}: If we select a subset of the collinear triples spanned by $A, B$ and $C$ of size at least $\delta n^2$ then in this subset, there will be a tic-tac-toe configuration. 
In this paper's proofs, when we refer to collinear triples (or simply triples or edges), these are triples in our selected subset of the geometrically collinear triples. We can remove them in the  proofs, even if the three points remain collinear triple in the geometric sense.
In the proof of Theorem \ref{main1}, we are going to use the geometric properties of the triples and the $(6,3)$ theorem of Ruzsa and Szemer\'edi twice (Theorem \ref{6:3} below). Because of the use of the Ruzsa-Szemer\'edi theorem, our bound on $n_0=n_0(\delta)$ is rather weak. 

%{On the other hand it is possible that tic-tac-toe configurations are unavoidable even with $n^{2-c}$ triples for some $c>0.$\footnote{There is no $n^{2-o(1)}$ type lower bound known for the $(9,6)$ case of the Brown-Erd\H{o}s-S\'{o}s conjecture. }}
As was mentioned earlier, it follows from the Brackenridge–Maclaurin theorem that if in Theorem \ref{main1} set $C$ is a set of collinear points, then three points in $A$ and three in $B$ are lying on the same conic (Figure  \ref{fig:Pascal}). We will analyze this case in Section \ref{collinear} of the paper.

\medskip
\subsection{Proof of Theorem \ref{main1}}\label{sec_proof}
During the proof, we will assume that set $B$ is the set of the middle points of the collinear triples. Since the three sets $A, B,$ and $C$ are mutually avoiding, any point in $B$ sees every point in $A$ in the same clockwise order and sees every point in $C$ in the same anticlockwise order in which we are using to label the points as $A=\{a_1,\ldots,a_n\}, C=\{c_1,\ldots,c_n\}.$ 
Now we sweep a line through an arbitrary point of $A$ anticlockwise to label $B=\{b_1,\ldots,b_n\}.$
With this labelling, the following conditions hold:

\begin{itemize}
    \item If $a_i,b_{j_1},c_{k_1}$ and $a_i,b_{j_2},c_{k_2}$ are collinear triples and $j_1<j_2$ then $k_1<k_2$
    \item If $a_{i_1},b_{j},c_{k_1}$ and $a_{i_2},b_{j},c_{k_2}$ are collinear triples and $i_1>i_2$ then $k_1<k_2$
   \item If $a_{i_1},b_{j_1},c_{k}$ and $a_{i_2},b_{j_2},c_{k}$ are collinear triples and $i_1<i_2$ then $j_1<j_2.$
\end{itemize}
%The figure below (fig \ref{fig:order}) illustrates the three conditions.

This ordering will be crucial in the proof since we will search for $(6,3)$ configurations (the second graph, labelled by {\em b,} in figure \ref{fig:evol}) where the vertices in the same set are close to each other. We will see that in this ordered geometric setting, $\delta n^2$ collinear triples determine some $\delta' n^2$ $(6,3)$ configurations on not too many, $\Delta n,$ close point-pairs ($\delta'$ and $\Delta$ are constants depending on $\delta$ only). Then we can apply the Ruzsa-Szemer\'edi theorem again to find a $(12,9)$ configuration, with nine edges spanned by twelve edges. This configuration contains tic-tac-toe (Figure \ref{fig:evol}).

\begin{figure}[h!]
\centering
\includegraphics[width=14cm]{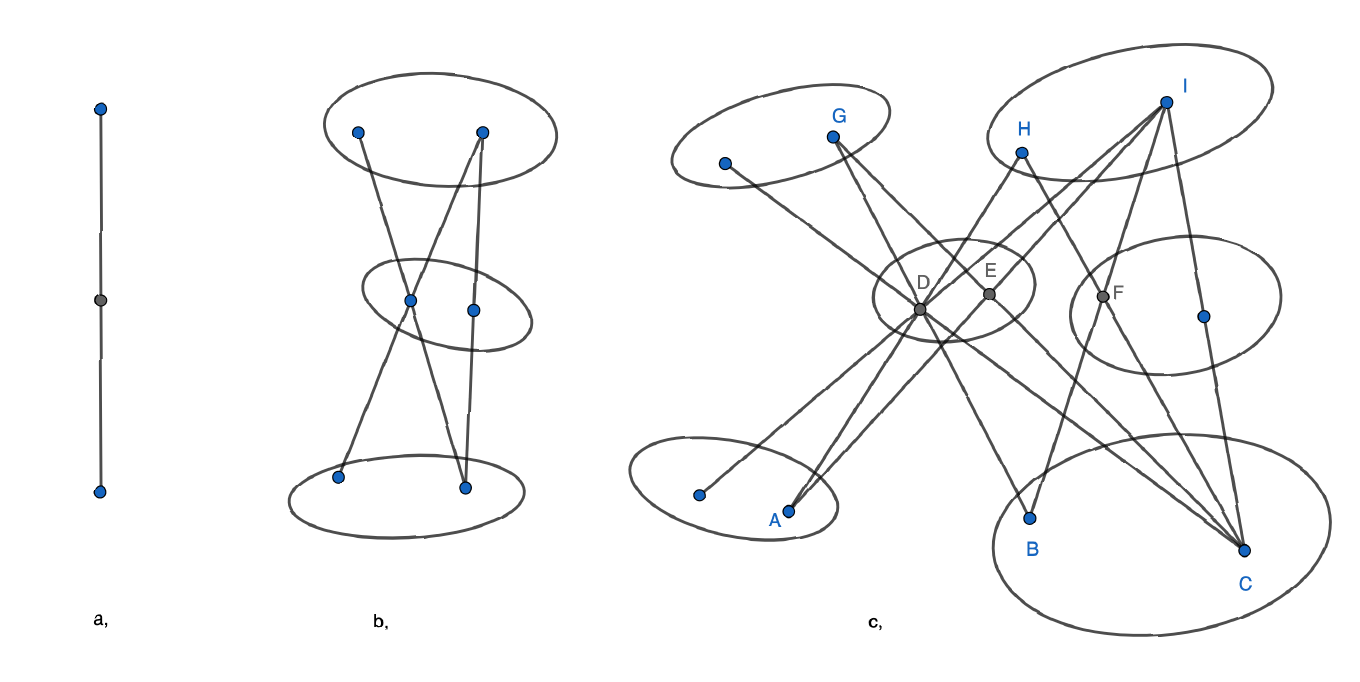}
\caption{From triples to tic-tac-toe (points $A, B,C, D,E, F,G, H,I$)}
\label{fig:evol}
\end{figure}

\subsubsection{Finding many skinny (6,3) configurations}
We begin this subsection by stating the theorem of Ruzsa and Szemer\'edi.

\begin{theorem}[Ruzsa-Szemer\'edi \cite{RSz}]\label{6:3}
For any $\epsilon>0$, there is a threshold, $\nu(\epsilon),$ such that if a 3-uniform hypergraph has $n\geq \nu(\epsilon)$ vertices and at least $\epsilon n^2$ edges, then there are six vertices spanning at least three edges of the hypergraph.
\end{theorem}

In our application, the hypergraph is a linear three-partite 3-uniform hypergraph since the edges a defined by collinear triples with one vertex in each set.
The only $(6,3)$ configuration which is possible is isomorphic to the second figure (b,) in Figure \ref{fig:evol}. 
\begin{definition}
A $(6,3)$ configuration in the sets $A=\{a_1,\ldots,a_n\}, B=\{b_1,\ldots,b_n\}, C=\{c_1,\ldots,c_n\}$ is called $N$-skinny if the spanning vertices are $a_i,a_j,b_k,b_\ell,c_u,c_v$ and 
$\max\{|i-j|,|k-\ell|,|u-v|\}\leq N.$
\end{definition}

\begin{lemma}\label{skinny}
For any $\delta>0$ there is a $\delta'>0$ and natural number $N$ such that $\delta n^2$ collinear $a,b,c$ triples ($a\in A, b\in B, c\in C$)  determine at least $\delta' n^2$ $N$-skinny $(6,3)$  configurations provided $n$ is large enough. Moreover, one can assume that the $(6,3)$ configurations are edge-disjoint; they don't share collinear triples.
\end{lemma}

\begin{proof}
Let us choose a large constant $M.$ We will determine the value of $M$ during the proof. It is a constant, so for the sake of simplicity, we can assume that $M$ divides $n.$ Let us partition the vertices of $A$ and $C$ into $n/M$ classes such that each partition class contains consecutive elements. $P_1=\{a_1,\ldots,a_M\}, P_2=\{a_{M+1},\dots,a_{2M}\},\ldots$ etc. and  $Q_1=\{c_1,\ldots,c_M\}, Q_2=\{c_{M+1},\dots,c_{2M}\},\ldots$ etc. 

\begin{definition}
For an $(i,j)$ pair, $b\in B$ is said to be {\em $\varepsilon$-good} if there are at least $\varepsilon M$ collinear triples $(a,b,c)$ where $a\in P_i$ and $c\in Q_j.$ 
By saying that an $(a,b,c)$ triple is $\varepsilon$-good we mean that $b$ is $\varepsilon$-good for the $(i,j)$ pair, where $a\in P_i$ and $c\in Q_j.$
\end{definition}

Let us count the number of $(b,i,j)$ triples where $b$ is $\varepsilon$-good for $(i,j)$. 

\medskip
For every $b\in B$ let us complete the following pruning process:

\begin{enumerate}
\item 
If the vertices of a $P_i$ contribute to less than $\frac{\delta}{4}M$ triples with $b$ and any $c\in C$ then let's remove these triples ($1\leq i\leq n/M$). Then we repeat the process with $Q_i$-s, if there are less than 
$\frac{\delta}{4}M$ vertices in a $Q_i$ forming a collinear triple with $b$ and some vertices of $A$ then remove the triples. Next, we check $P_i$-s again, removing triples if there is less than $\frac{\delta}{4}M$ triples with the vertices in this partition class. Then we check $Q_i$-s again and repeat this alternating process as long as possible. The total number of removed triples is denoted by $T_b$. It is less than $2\frac{n}{M}\frac{\delta}{4}M=\frac{\delta}{2}n $. 
\item
  Let us call triples $(a,b,c)$ $a\in P_i$ and $c\in Q_j$ with point $b\in B$ to be {\em weak} if the total number of collinear triples  $(a',b,c')$ $a'\in P_i$ and $c'\in Q_j$ is less than $\frac{\delta}{8}M,$ or, equivalently,  $b$ is not $\frac{\delta}{8}$-good for $(i,j).$  Remove all weak triples.
\end{enumerate}

We have to show that after the second step, we still have many triples if there were many triples through $b$. At this point, we have to use the geometric properties of sets $A$ and $C.$
If there is a collinear triple $(a,b,c),$ $a\in P_i$ and $c\in Q_j$ with point $b\in B$, then we can't conclude that there are many triples between $P_i$ and $Q_j$ through $b$ even in this geometric setting (see Figure \ref{fig:off}).  

\begin{figure}[h!]
\centering
\includegraphics[width=14cm]{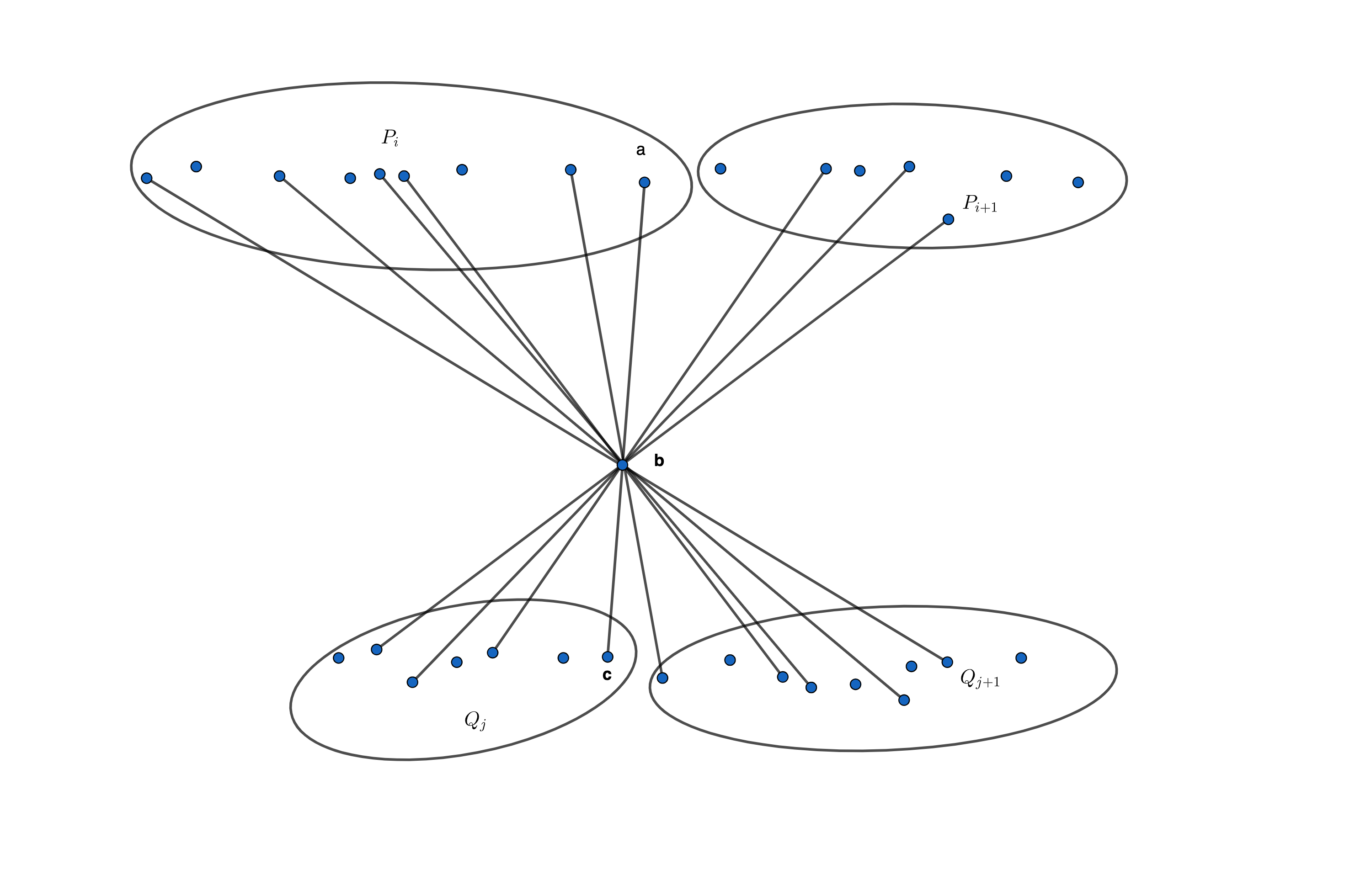}
\caption{There is only one collinear triple between $P_i$ and $Q_j$ }
\label{fig:off}
\end{figure}

On the other hand, a collinear triple guarantees that there are many triples either between $P_i$ and $Q_j$ or in their neighbourhood.
If there is a collinear triple between $P_i$ and $Q_j$ via $b\in B$ then let $V_b=\{a_{i_1},\ldots, a_{i_m}\}\subset P_i,$ (${i_1}<\ldots<{i_m}$) denote the set of points in $P_i$ where there is another point in $C$ forming a triple with $b.$ Similarly, $W_b=\{c_{j_1},\ldots, c_{j_\ell}\}\subset Q_j,$ (${j_1}<\ldots<{j_\ell}$) denote the set of points in $Q_j$ where there is another point in $A$ forming a triple with $b.$ 
Because of the first step, we  know that both $m$ and $\ell$ are at least $\frac{\delta}{4}M.$ 

\medskip
For any triple, $(a_{i_s},b,c_{j_t}),$ at least $\min\left(s,\frac{\delta}{4}M\right)$ triples $(a_{i_u},b,x),$ where $x\in C$ and $u\leq s,$ are triples between $P_i$ and $Q_j\cup Q_{j+1}.$
\footnote{Let us suppose that $P_0=Q_0=P_{\frac{n}{m}+1}=Q_{\frac{n}{m}+1}=\emptyset$} When $u\geq s$ then at least $\min\left(m-s,\frac{\delta}{4}M\right)$ triples are 
between $P_i$ and $Q_j\cup Q_{j-1}.$ On the other side, for the triple $(a_{i_s},b,c_{j_t}),$ at least $\min\left(t,\frac{\delta}{4}M\right)$ triples $(y,b,c_{j_v})$, where $y\in A$
and $v\leq t$, are triples between $Q_j$ and $P_i\cup P_{i+1}.$ When $v\geq t$ then at least $\min\left(\ell-v ,\frac{\delta}{4}M\right)$ triples are between $Q_j$ and $P_i\cup P_{i-1}\}.$ 

If $(a_{i_s},b,c_{j_t})$ is weak then both $a_{i_s}$ and $c_{j_t}$ should be among the first $\frac{\delta}{8}M$ points or both are among the last $\frac{\delta}{8}M$ points of $V_b$ and $W_b.$ More precisely $$\min(s+t,(m+\ell-s-t))\leq \frac{\delta}{8}M.$$ 
After the first step of the pruning process, let's check the triples $(a_\alpha ,b,c_\beta ),$ $a_\alpha\in A, c_\beta\in C,$ starting with the least index $\alpha$ in increasing order. By the previous observations, the longest sequence of weak triples is less than $\frac{\delta}{8}M$ always followed by at least $\frac{\delta}{8}M$ many $\frac{\delta}{8}$-good triples (if there are triples left). 
If the number of triples through $b$ after the first step is denoted by $\Gamma_b$, then the number of weak triples is at most $\frac{\Gamma_b}{2}+ \frac{\delta}{8}M.$

\noindent
After the pruning process is finished for all $b\in B,$ the total number of removed triples is not more than

\begin{equation}\label{badtriples_0}
\sum_{b\in B}\left(T_b+\frac{\Gamma_b}{2}+\frac{\delta}{8} M\right).
\end{equation}

\noindent
Since $\sum_{b\in B}(T_b+\Gamma_b)=\delta n^2$ and $T_b<\frac{\delta}{2}n$, we have the following upper bound on the number of removed triples:

\begin{equation}\label{badtriples}
\sum_{b\in B}\left(\frac{T_b}{2} +\frac{T_b+\Gamma_b}{2}+\frac{\delta}{8} M\right)\leq \frac{\delta}{4}n^2 + \frac{\delta n^2}{2} + \frac{\delta}{8} Mn=
\frac{3\delta n^2}{4} + \frac{\delta}{8} Mn.
\end{equation}

\medskip
\noindent
We still have about the quarter of the original triples, $\frac{\delta}{4}n^2-O(n),$ and now every triple is $\frac{\delta}{8}$-good for the relevant $(i,j)$ pair. 
For reference in later applications, we record our partial result in the following claim. 

\begin{claim}\label{calc}
For any $\delta>0$ there is a constant $M$ such that under the conditions of Theorem \ref{main1} at least $\frac{\delta}{4}n^2 -\frac{\delta}{8} Mn$ triples are $\frac{\delta}{8}$-good. 
\end{claim}

In the calculations leading to Claim \ref{calc}, the parameter $M$ was set to be a constant, however, the bounds hold (with a larger error term)  if $M$ is a slow-growing function of $n$.

\noindent
Note that up to this point, the only property we used for the pointset $B$ is that no line determined by the points of $A$ or by the points of $C$ intersects the convex hull of $B,$ so Claim \ref{calc} holds under the weaker conditions that $A$ and $C$ are mutually avoiding and $A$, and $C$ are avoiding $B.$ 

\medskip
In the remaining part of the proof of Theorem \ref{main1}, we ignore the linear error term, it does not affect the arguments. As the last pruning step, we check the number of the remaining collinear triples between the $P_i, Q_j$ pairs for all $1\leq i,j\leq n/M.$ If the number of triples is less than $\frac{\delta}{8}M^2$ then remove all triples between $P_i$ and $Q_j.$ After this, we have left with at least $\frac{\delta}{8}n^2$ triples, and whenever there is a triple between two partition classes $P_i$ and $Q_j$ it is  $\frac{\delta}{8}$-good for $(i,j),$ and there are at least $\frac{\delta}{8}M^2$ triples between $P_i$ and $Q_j.$

\medskip
We can apply the Ruzsa-Szemer\'edi theorem for triple systems between all $P_i, Q_j$ pairs where we have at least one triple left between them. As we noted before, there are at least $\frac{\delta}{8}M^2$ triples between them, and the number of points from $B$ forming the triples is at most 
\[
\frac{M^2}{\frac{\delta}{8}M}  =\frac{8}{\delta}M
\]

By Theorem \ref{6:3}, we can set $M$ large enough such that if the number of triples is at least $\frac{\delta}{8}M^2$ on $\left(\frac{8}{\delta}+2\right)M$ vertices, then it contains a $(6,3)$ configuration. It is satisfied if we choose $M=\nu(\delta^3/128).$ We continue to work with this parameter $M.$
The number of triples is at least $\frac{\delta}{8}n^2$, so there are at least $\frac{\delta}{8M^2}n^2$ distinct $(i,j)$ pairs so that there is a $(6,3)$ configurations spanned by the triples between $P_i$ and $Q_j.$ For every $1\leq i\leq n/M$ we select two points from $P_i$ and two points from $Q_i$ such that the selected pairs span many $(6,3)$ configurations (together with the points of $B.$) Let's choose the pairs independently at random. The probability that both pairs of points of a $(6,3)$ configuration from $A$ and $C$ have been selected is $\binom{M}{2}^{-2}.$
So, there is a way to select $4n/M$ points, two from every partition class of $A$ and $C,$ such that the pairs which were chosen span at least $\frac{\delta}{4M^4}n^2$ distinct $(6,3)$ configurations. 
If the spanning vertices of such a configuration are denoted by $a_i,a_j,b_k,b_\ell,c_u,c_v$ then
$\max\{|i-j|,|u-v|\}\leq M.$

\medskip
What is left in the proof of Lemma \ref{skinny} is to show that there is a constant $N,$ such that many of the selected configurations satisfy $|k-\ell|\leq N.$ For this, we are going to use the ordering (labelling) of $B.$ In a $(6,3)$ configuration, there is one point in $B$ which is in two collinear triples, and the other in one triple. In at least half of the configurations, the order of the two types of points is the same. Let us suppose that the doubhas the smaller subscript in most of the  earlier selected triples and keep these  configurations only. 
We have at least $\frac{\delta}{8M^4}n^2$ configurations left.
Every $(6,3)$ configuration has its pair of points in $C$ in one of the $n/M$ pairs selected in the random process. If the number of configurations on a selected pair is less than half the average when it is less than $\frac{\delta}{16M^3}n,$ then throw them away. We have at least $\frac{\delta}{16M^4}n^2$ configurations left. Note that in any two configurations with the same pair of points in $C,$ where the vertices are denoted by $a_i,a_j,b_k,b_\ell,c_u,c_v$ and $a_{i'},a_{j'},b_{k'},b_{\ell'},c_u,c_v,$ the intervals $[k,\ell]$ and $[k',\ell']$ are disjoint, since the pairs, $a_i,a_j$ and $a_{i'},a_{j'}$ are in different partition classes. (Figure \ref{fig:two} ). 

\begin{figure}[h!]
\centering
\includegraphics[width=6cm]{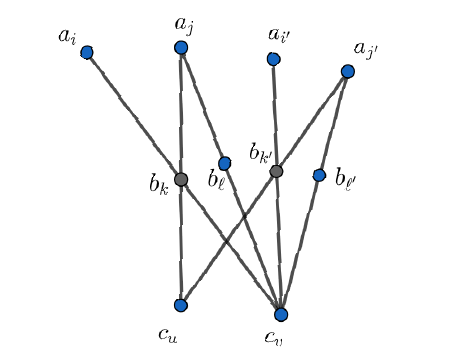}
\caption{The intervals $[k,\ell]$ and $[k',\ell']$ are disjoint }
\label{fig:two}
\end{figure}

Since these intervals of indices are disjoint, no pair in $C$ can hold more than $\frac{\delta}{32M^3}n,$ configurations where the two points  $b_k,b_\ell\in B$ are far, when $|k-\ell|\geq \frac{32M^3}{\delta}.$
This proves the lemma with parameters $N=\frac{32M^3}{\delta}$ and $\delta'=\frac{\delta}{16M^4}n^2,$ where $M=\nu(\delta^3/128).$
\qed
\end{proof}

\subsubsection{Using Ruzsa-Szemer\'edi again}

To complete the proof of Theorem \ref{main1}, we define a 3-uniform hypergraph, where the edges are the skinny $(6,3)$ configurations we are left with after the last step in the previous section. From the pointsets, $A$ and $B$, the vertices of the new hypergraph are the $2n/M$ pairs we selected in the random process. In pointset $B$, let's consider all pairs $b_k,b_\ell\in B,$ where $|k-\ell|\leq N.$ For each such pair, we can assign an integer, the number of $(6,3)$ configurations containing this pair of points. We follow a greedy algorithm to select vertex disjoint pairs contributing to many skinny $(6,3)$ configurations. 

\begin{enumerate}
    \item Select a pair $b_k,b_\ell$ with the highest assigned value, say $v$.
    \item Remove all pairs $b_{k'},b_{\ell'}$ where the $k=k'$ or $\ell=\ell'.$ There are less than $4N$ such pairs, and all removed pairs had a value less than $v.$
    \item If there are pairs left, go back to step one.
\end{enumerate}

The selected pairs will be the middle vertices of the triples (edges). Altogether we have at least $\frac{\delta'}{4N}n^2$ edges (triples) on at most $2n/M+n/2$ vertices. 
We can apply Theorem \ref{6:3} to conclude that if $n\geq \nu(\frac{\delta'}{4N})$ then it contains the $(12,9)$ configuration we were looking for (picture $c,$ in Figure \ref{fig:evol}).

\section{The Brown-Erd\H{o}s-S\'{o}s conjecture for collinear triples}
One of the most famous open problems in extremal combinatorics is the following:

\begin{conjecture}[Brown-Erd\H{o}s-S\'{o}s \cite{BES}]\label{conj:BES}
	Fix $m \geq 6$. For every $c>0$, there exists a threshold $N = N(c)$ such that if $H_n^{(3)}$ is a 3-uniform hypergraph on $n$ vertices  with $|S| \geq cn^2$, edges there exists a subset of $m$ vertices of $H_n^{(3)}$ which spans at least $m-3$ edges. (i.e. the hypergraph contains an $(m,m-3)$ configuration)
\end{conjecture}

About 50 years later, the only known case of Conjecture \ref{conj:BES} is when $m=6,$ Theorem \ref{6:3} above. We refer to \cite{SS} and \cite{CGLS} for some estimates and reviews of related works.  Conjecture \ref{conj:BES} has resisted all proof attempts in its general form. However, some special cases exist when even stronger statements can be proved. When the edges of the hypergraph are defined by a finite group with edges of the form $(a,b,a+b)$, the underlying structure helps to find vertices spanning many edges. There are recent works proving such results like in \cite{So2},\cite{CS},\cite{L},\cite{C}, and \cite{NST}. 

We proved the existence of $(12,9)$ configurations in large dense systems of collinear triples between three mutually avoiding pointsets. Removing points from this arrangement we have $(11,8),(10,7),(9,6)$ configurations, matching the numbers of the B-E-S conjecture. For larger values, we can get denser configurations, similar to the case of triples from finite groups. One way to prove such results is to iterate the method we used in the proof of Theorem \ref{main1}. We can find many skinny $(12,9)$ configurations and consider them as edges of a 3-uniform graph on some four-tuples of points as vertices, and then apply Ruzsa-Szemer\'edi again. This way we would get $3\cdot 9$ edges on $2\cdot 12$ vertices, a $(24,27)$ configuration. One would get $(3^{k+1}, 2^k3)$ configurations by further iterations. But this method involves the iteration of the Ruzsa-Szemer\'edi theorem, which gives a very poor bound. On the other hand, we can use some elements of the previous proof to get better bounds than what would follow from the B-E-S conjecture, even in sparse systems.

\begin{definition}
A $k$-complete collinear triple arrangement (or $k$-system in short) consists of $k$ disjoint $k$-element pointsets in $A,$ denoted by
$A_1=\{a_1^{(1)},\ldots,a_k^{(1)}\},\ldots, A_k=\{a_1^{(k)},\ldots, a_k^{(k)}\}$ and $k$ disjoint $k$-element pointsets in $C,$ denoted by
$C_1=\{c_1^{(1)},\ldots, c_1^{(k)},\ldots, c_k^{(k)}\}$ such that for every $1\leq i,j\leq k$ there is a point $b_{i,j}\in B$
such that for every $1\leq \ell \leq k$ the points $a_\ell^{(i)},$ $b_{i,j}$ and $c_{k-\ell+1}^{(j)}$  form a collinear triple.
\end{definition}

\begin{figure}[h!]
\centering
\includegraphics[width=6cm]{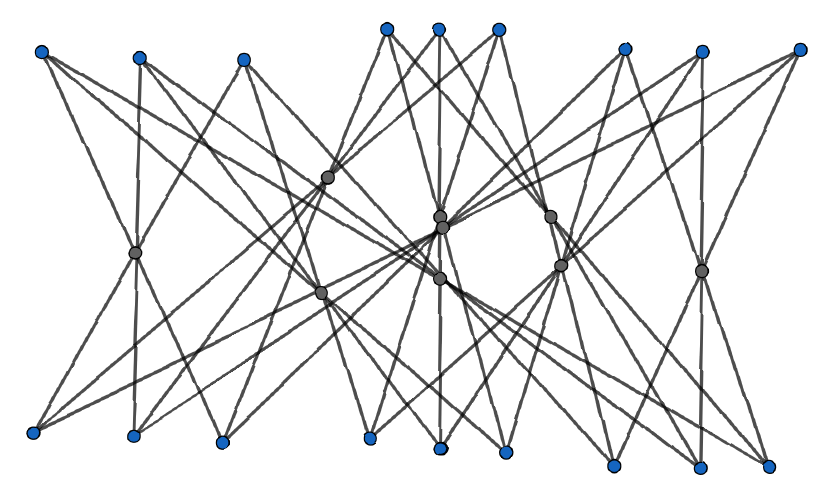}
\caption{This is a $3$-complete collinear triple arrangement (a 3-system) }
\label{fig:3-sys}
\end{figure}

A $k$-system has at most $3k^2$ points (some $b_{i,j}$ points might be the same) and $k^3$ edges.

\begin{theorem}\label{main2}
For any $k\in \mathbb{N}$ there are positive constants $\gamma_k >0,\mu_k>0$ and a threshold $n_0(k)$ such that the following holds. Let $A$ and $C$ be mutually avoiding $n$-element pointsets in the plane such that no line determined by $A$ or determined by $C$ intersects the convex hull of the $n$-element pointset $B$. Let us suppose that there are at least $ \mu_k n^{2-\gamma_k}$ collinear $(a,b,c)$ triples, where $a\in A, b\in B,$ and $c\in C.$ If $n\geq n_0,$ then there is a $k$-system among the collinear triples.
\end{theorem}

\begin{proof}
We will use the calculations and notations leading to Claim \ref{calc} in the proof of Theorem \ref{main1}. There are at least $\frac{\delta}{4}n^2 -\frac{\delta}{8} Mn$ triples that are $\frac{\delta}{8}$-good among $\delta n^2$ triples. Set $k=\frac{\delta}{8}M, \mu_k=8k$ and $M=n^{\gamma_k}$. 
We will specify the value of $0<\gamma_k<1$ later. If it is small enough, we can still assume that $M$ divides $n,$ so the calculations before Claim \ref{calc} can be repeated with a non-constant $M$.
 (If the number of triples in the arrangement was $T,$ then we can throw away a few points and less than $2T/n^{1-\gamma_k}$ triples adjacent to them before starting the proof)
With this setup  $\delta=\frac{8k}{n^{\gamma_k}}$ and 
at least $2kn^{2-\gamma_k}-kn$ triples are $\frac{k}{M}$-good 
since $\mu_k n^{2-\gamma_k}=\delta n^2.$

We now estimate the number of $k$-branches, the building blocks of $k$-systems. The $k$-branches are spanned by the $P_i, Q_j$ partition classes for some $1\leq i,j\leq n^{1-\gamma_k}.$ 
They are formed by $k$ triples, $(a,b,c)$ sharing the same $b\in B$  where $a\in P_i$ and $c\in Q_j.$ There are at least $2n^{2-\gamma_k}-n$ such branches. 
Let's choose $k$ points from each $P_i$ and $Q_j$ independently at random. If the number of $k$-branches between $P_i$ and $Q_j$ is denoted by $K_{ij}$ then the probability of selecting the $2k$ points of a $k$-branch
in $A$ and $B$ is $K_{ij}\binom{M}{k}^{-2},$ so there is a selection of $k$-tuples where the number of $k$-branches is at least

\begin{equation}\label{k-branch}
\begin{split}
\sum_{1\leq i,j\leq n^{1-\gamma}} K_{ij}\binom{M}{k}^{-2}\geq \left(2n^{2-\gamma_k}-n\right)\binom{M}{k}^{-2} \\
%\geq \frac{(2n^{2-\gamma_k}-n^{1+\gamma_k})k!}{e\cdot n^{k\gamma_k}}
\geq k!^2\left(2n^{2-(2k+1)\gamma_k}-n^{(1-2k\gamma_k)}\right)
\geq k!^2n^{2-(2k+1)\gamma_k}.
\end{split}
\end{equation}

Let us define an auxiliary bipartite graph where the vertices are the selected $k$-tuples of points, and there is an edge between two if there was a $k$-branch selected on the $2k$ points.
This graph has $n^{1-\gamma_k}$ vertices in each vertex class and at least ${k!}^2n^{2-(2k+1)\gamma_k}$ edges. By the K\H{o}v\'ari-S\'os-Tur\'an theorem \cite{KST} if
\[
(1-\gamma_k)\left(2-\frac{1}{k}\right)\leq 2-(2k+1)\gamma_k \quad \text{i.e. when}\quad \gamma_k\leq \frac{1}{2k^2-k+1,}
\]
then the graph contains a complete bipartite graph $K_{k,k}$ which proves the existence of a $k$-system among the triples.
\qed
\end{proof}

With the same $\gamma_k$ exponent, by increasing the constant multiplier, $\mu_k,$ a larger bipartite graph,  $K_{k,t}$ will appear in the bipartite graph. Using a larger $t$, one can guarantee a $k$-system where all $b\in B$ vertices are distinct, so the system contains exactly $3k^2$ points. It is a particular $(3k^2,k^3)$ configuration, which is an 
unavoidable collinear triple configuration, even in sparse systems. 

\section{When one set of points is collinear}\label{collinear}
In this section, we assume that set $C$ consists of $n$ collinear points, the convex hulls of $A$ and $B$ are not intersected by the line of $C$, and no line determined by the points of $A$ intersects the convex hull of $B$.
Let us suppose that that there are $\delta n^2$ collinear $(a,b,c)$ triples, where $a\in A, b\in B,$ and $c\in C.$
In the case when $A$ is also a set of collinear points, Elekes proved that a positive fraction of the points of $B$ should also be collinear \cite{El}.
Elekes and Szab\'o proved the stronger statement that if the points of $A\cup C$ are on a conic, then a positive fraction of the points of $B$ is collinear (Theorem 5.2 in \cite{ESz}).
Our main result in this section is to show that if $C$ is collinear, then---under some additional conditions---many points of $A$ and  $B$ are on a conic. 
Let's recall another conjecture of Elekes, similar to the earlier ones:

\begin{conjecture}[Elekes \cite{E3}]\label{Elekes6} For every $C > 0$ there is an $n_0 = n_0 (C)$ with the following property. 
If $A\subset \mathbb{R}^2,$ $|A|\geq n_0$ and the number of directions determined by $A$ is at most $C|A|$ then $A$ contains six points of a (possibly degenerate) conic. 
(As usual, a pair of lines is considered a degenerate conic.) 
\end{conjecture}

\noindent
In his classical paper {\em Sums vs Products} \cite{E_SP} Elekes writes the following about his conjecture: 

\medskip
{\em ``It is very likely that Conjecture \ref{Elekes6} holds for any number in place of six, for 
$|A|$ large enough. It was pointed out by M. Simonovits that one cannot expect $c^*|A|$ conconic points in general (as shown by a square lattice.) However, some $|A|^\alpha$ such points may exist for a suitable $\alpha=\alpha(C)$. Perhaps even $c^*|A|$ can be found, provided that $A$ is the vertex set of a convex polygon.''}

\medskip
\noindent
We will show a stronger, density version of this later statement.

\begin{theorem}\label{cor:elekes}
For any $D>0$ and  $c>0$, there is a constant, $\delta>0,$ and a threshold $n_0=n_0(c,D)$, so the following holds. Let $A$ be a set of $n$ points in the plane in convex position (the vertex set of a convex polygon) such that at least $cn^2$ point pairs determine at most $Dn$ distinct directions. If $|A|\geq n_0$, then at least $\delta n$ points of $A$ are on a conic. 
\end{theorem}
Theorem \ref{cor:elekes} is a direct corollary of Lemma \ref{main3} below. Before proving Elekes' conjecture, we recall a definition from projective geometry.

\begin{definition}
Two triangles are said to be perspectives from a line if the three points of intersections of corresponding sides all lie on a common line.
\end{definition}

For some restricted arrangements of points in a convex position, we use the following standard definition 

\begin{definition}
A set of $k$ points is a $k$-cap relative to a line $\ell$ if the points are in convex position and the convex hull of the $k$ points and their orthogonal projections onto $\ell$ contains all $2k$ points on the boundary.
\end{definition}

The points of a $k$-cap have a natural ordering based on the order of their projections onto $\ell$. Using the ordering of the projection points from left to right (say), we write $p_1<p_2<\ldots<p_k$ to indicate the ordering of the points $p_i$.

\begin{figure}[h!]
\centering
\includegraphics[width=14cm]{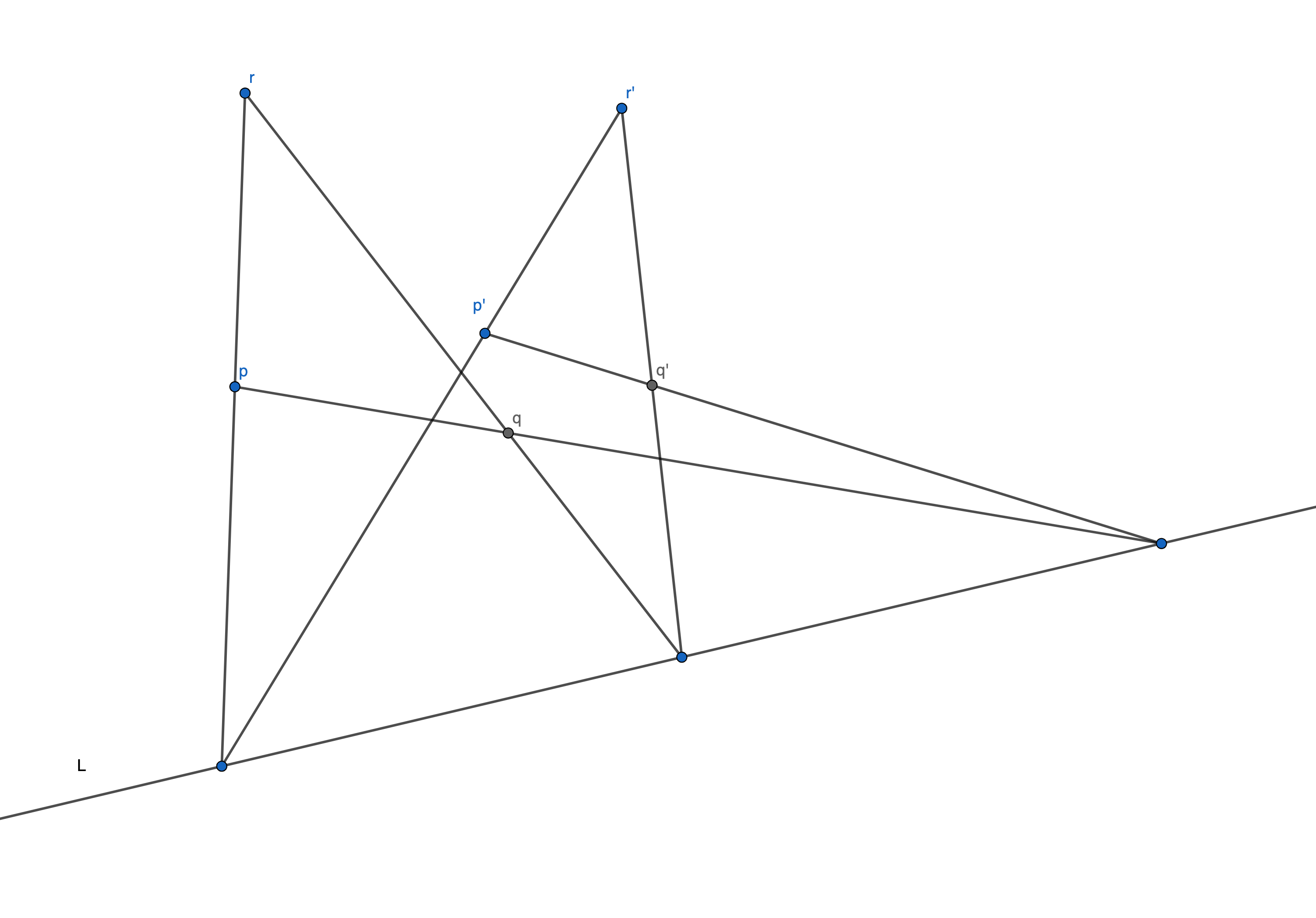}
\caption{ $(p,q,r)$ and $(p',q',r')$ are perspective from $L$ }
\label{fig:sim}
\end{figure}

\noindent

\begin{claim}
    Let $p_1<p_2<p_3<p_4<p_5<p_6$ be a $6$-cap relative to a line $\ell$. Then the two triangles spanned by $p_1,p_2,p_3$ and $p_4,p_5,p_6$ are not perspective from $\ell$.
\end{claim}

\begin{proof}
    Consider four points forming a $4$-cap relative to $\ell$. Denote them as  $q_1<q_2<q_3<q_4$. The intersection point of the two lines containing $\overline{q_1q_2}$ and $\overline{q_3q_4}$ is denoted by $x$. Let us suppose $x\in\ell$ and let $q_i'$ denote the projection of $q_i$ onto $\ell$. We consider two cases
    \begin{itemize}
        \item $x>q_2'$. In this case, $q_2\quad (<q_3)$ is not on the boundary of the convex hull of the points\\ $q_1,q_2,q_3,q_4,q_1',q_2',q_3',q_4'$ which contradicts the assumption that the four points formed a $4$-cap relative to $\ell$.
        \item $x<q_3'$. In this case, $q_3\quad (>q_2)$ is not on the boundary of the convex hull of the points\\ $q_1,q_2,q_3,q_4,q_1',q_2',q_3',q_4'$ which is a contradiction as in the previous case.
        \end{itemize}
    \qed
    
    \end{proof}

\medskip
In the first section, we proved Theorem \ref{main1}, which implies that if $C$ is a set of collinear points, then many 6-tuples of points from $A\cup B$ are on a conic. Now we have all the tools to show that a positive fraction of the points in $A\cup B$ is on a conic under some additional conditions. 

\begin{lemma}\label{main3}
For any $\delta>0$ constant there is a threshold $n_0=n_0(\delta)$ and $\gamma=\gamma(\delta)$ so the following holds. Let $A,B,C$ be $n$-element pointsets in the plane. Assume
that $n > n_0$ and
\begin{enumerate}
    \item $C$ is a set of collinear points,
    \item the convex hulls of $A$ and $B$ are not intersected by the line of $C,$
    \item no line determined by the points of $A$ intersects the convex hull of $B$. This condition allows us to label the elements of $A$ as $a_1,a_2,\ldots,a_n$, in the order they appear in a clockwise sweep through a point of $B$. (As we did at the beginning of the proof of Theorem \ref{main1} in section \ref{sec_proof})
    \item no three points are collinear in $A$, 
    \item and no two triples, $a_i,a_j,a_k$ and $a_\ell,a_f,a_g$ are perspective from the line of $C$ if\\ $i<j<k<\ell<f<g.$
    \item there are at least $\delta n^2$ collinear $(a,b,c)$ triples, where $a\in A, b\in B$ is between $a$ and $c,$ where $c\in C.$ 
\end{enumerate}
Then at least $\gamma n$ points of $ B$ are on a conic.
\end{lemma}

\medskip
\noindent
Before the proof, we give an example to show that something like conditions 4 and 5 is necessary for the conclusion. The example consists of two copies of a projection of the $d$-dimensional integer grid $[m]^d,$ where $m\sim (n/2)^{1/d},$ and the line at infinity, $C$ is the points of directions. Take a generic projection of $[m]^d$ onto the plane, and denote its points by $G$. Choose a translation of $G,$ denoted by $G',$ such that $G$ and $G'$ are mutually avoiding pointsets. We can select many point pairs, one from $G$ and one from $G'$, such that these pairs determine a few distances only. Define a set $S\subset [m]^d$ as points with all coordinates being far from the extremes. A point, $s=(s_1,s_2,\ldots,s_m),$ is in $S$ iff $\frac{m}{4}\leq s_i\leq \frac{3m}{4}$ for every $1\leq i\leq m.$ With these parameters, we have 

\[|S|\sim \left(\frac{m}{2}\right)^d\sim\frac{n}{2^{d+1}}.\]

\noindent
The points of the image of $S$ in $G$ and $G'$ are denoted by $G_S$ and $G'_S.$ Any direction determined by two points $p\in G_S$ and $q\in G'_S$ appears with multiplicity at least $\left(\frac{m}{2}\right)^d$ between points of $G$ and $G'.$ Select these directions (points) for set $C.$ For the size of $C$ the following rough bound holds:

\[
|C|\leq {\left(\frac{n}{2}\right)^2}\left({\frac{n}{2^{d+1}}}\right)^{-1}=2^{d-1}n.
\]

\noindent
For a fixed constant, $d\in \mathbb{N},$ all conditions of Lemma \ref{main3} hold except conditions 4 and 5, and there is no conic containing more than $n^{1/d}$ points from $G\cup G'.$
In this construction, there are many similar triangle pairs and collinear triples.

\subsection{The proof of Lemma \ref{main3}}

In the proof, we refer to the calculations preceding Claim \ref{calc}. Based on the collinear triples, we will define a bipartite graph formed by 3-branches in a similar way as we did in the proof of Theorem \ref{main2}. This bipartite graph is dense, so we will find a $K_{2,\ell}$ subgraph where $A$ contains two vertices (quadruples) and $C$ contains $\ell\sim cn$ vertices. We will show then that the points of the corresponding $b\in B$ points are on a low-degree algebraic curve. We will follow the exponents of $\delta$ during the calculations in case one is interested in conclusions similar to Lemma  \ref{main3} when $\delta=n^{\varepsilon}$ for some $\varepsilon >0.$ 
%We can assume w.l.o.g. that every $b\in B$ is in at least $\frac{\delta}{2}n$ triples.

\subsubsection{The auxiliary bipartite graph}
We will use the same parameters we had in the previous sections. By Claim \ref{calc} at least 

\begin{equation}\label{goodtriples}
    \frac{\delta}{4}n^2 -\frac{\delta}{8} Mn
\end{equation}
triples are $\frac{\delta}{8}$-good for the relevant $(i,j)$ pairs. We choose the parameter $M$ to be the least integer such that $\lfloor\frac{\delta M}{8}\rfloor\geq 3, $ so if there is a triple between a $P_i,Q_j$ pair after  the pruning process before Claim \ref{calc}, then there is at least one 3-branch between $P_i$ and $Q_j.$ We choose independently at random a triple of points in each $P_i$ and $Q_j$ ($1\leq i,j\leq n/M$). Since $M\leq \frac{24}{\delta}+1,$ a constant nicely depending on $\delta$, we can ignore the linear term in (\ref{goodtriples}) which is at most about $3n$. There are at least $\frac{\delta}{12}n^2$ distinct (but not necessarily vertex disjoint) 3-branches. The expected number of selected 3-branches - when both the top and bottom triples were selected - is at least 

\begin{equation}
    \binom{M}{3}^{-2}\frac{\delta}{12}n^2 >c_1\delta^7n^2,
\end{equation}

\noindent
with a universal constant $c_1>0.$ Now we are ready to define the auxiliary bipartite graph, $G(P,Q)$. The vertices are the selected triples in $P_i$-s and $Q_j$-s. There are $2\times n/M$ such triples, so both vertex classes of the bipartite graph $G(P,Q)$ have size about $\frac{\delta n}{24},$ denoted by $N.$ 
Two vertices, given by $P_i$ and $Q_j$, are connected by an edge if the selected triples span a 3-branch. 

There is a choice of triples such that the number of edges in $G(P,Q)$ is at least the expected number of edges in the random selection, it is at least $c_2\delta^5N^2.$ If $n$ is large enough then there are two vertices in $P$ with at least $c_3\delta^{10}N$ common neighbours in $Q.$ Indeed, the number of 3-paths with the middle vertex in $Q$ is

\[
\sum_{v\in Q}\binom{deg(v)}{2}\geq N\binom{c_2\delta^5N}{2}\approx \frac{c_2^2\delta^{10}N^3}{2}   .
\]

\noindent
There is a pair, $v,w\in P$, out of the $\binom{N}{2}$ vertex-pairs wich are the end-vertices of at least $\sim c_2^2\delta^{10}N$ different 3-paths. The two vertices, $v$ and $w$ have at least $\sim c_2^2\delta^{10}N$ common neighbours.

\medskip

Let us denote the two triples of points corresponding to $v$ and $w$ in $A$ by $(\alpha_1,\beta_1),(\alpha_2,\beta_2),(\alpha_3,\beta_3)$\\
and $(a'_1,b'_1),(a'_2,b'_2),(a'_3,b'_3).$ Since we selected one triple from the $P_i$ sets, 
we can assume that\\ $\alpha_1<\alpha_2<\alpha_3<a'_1<a'_2<a'_3.$ 
The corresponding triples in $Q$ are denoted by $(s'_1,t'_1,u'_1),\ldots ,(s'_m,t'_m,u'_m)$ where $m\geq c_3\delta^{10}N.$

\subsubsection{Finding the algebraic curve in $B$}
To simplify the remaining calculations, let us apply a projective transformation, denoted by $T_1$, which sends the first three points, $(\alpha_1,\beta_1),(\alpha_2,\beta_2),(\alpha_3,\beta_3),$ to the pre-set coordinates $(1,1),(0,1),(0,2)$, and  the image of line $C$ is the $x$-axis. A projective transformation could change convexity and other geometric relations among the points, but it preserves collinearity, and this is the only property we are using below.

%$$T_1: (\alpha_1,\beta_1),(\alpha_2,\beta_2),(\alpha_3,\beta_3)\Longrightarrow (1,1),(0,1),(0,2) \quad \text{(Figure \ref{fig:masodik})}$$ 

\noindent

\begin{figure}[h]
\begin{tabular}{ll}
\includegraphics[width=8cm]{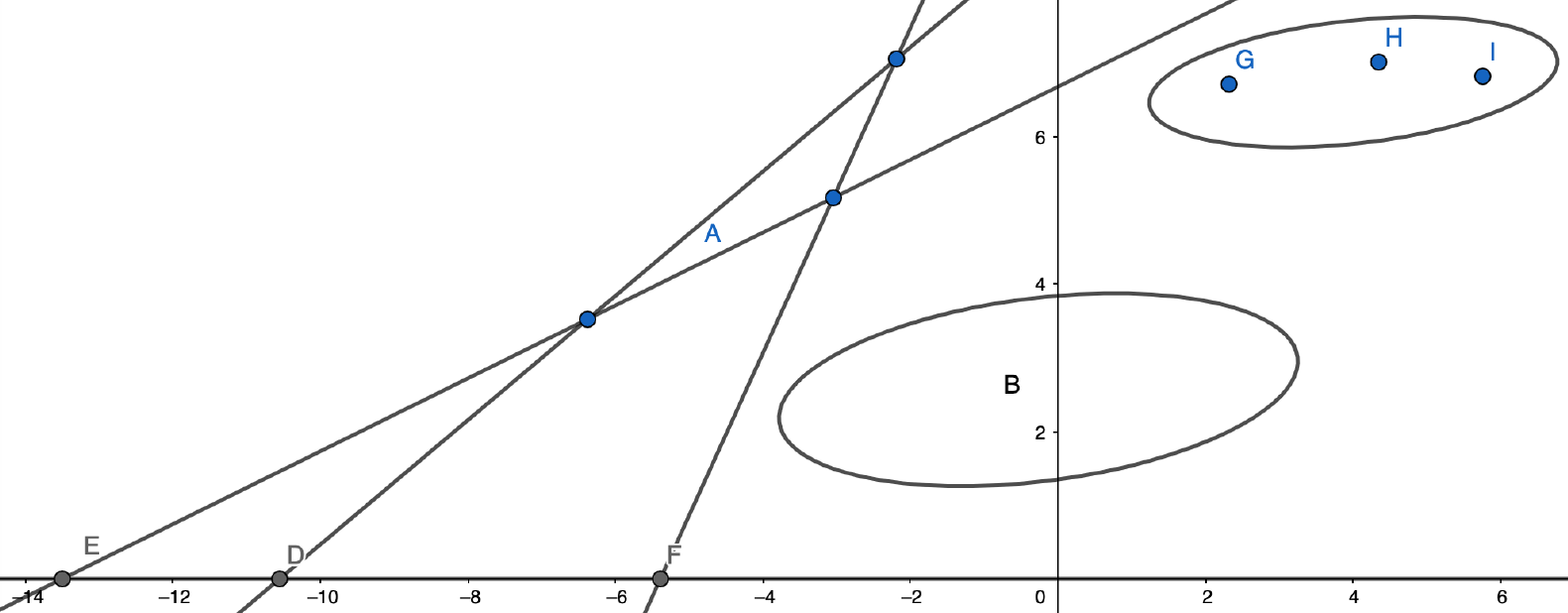}
&
\includegraphics[width=6cm]{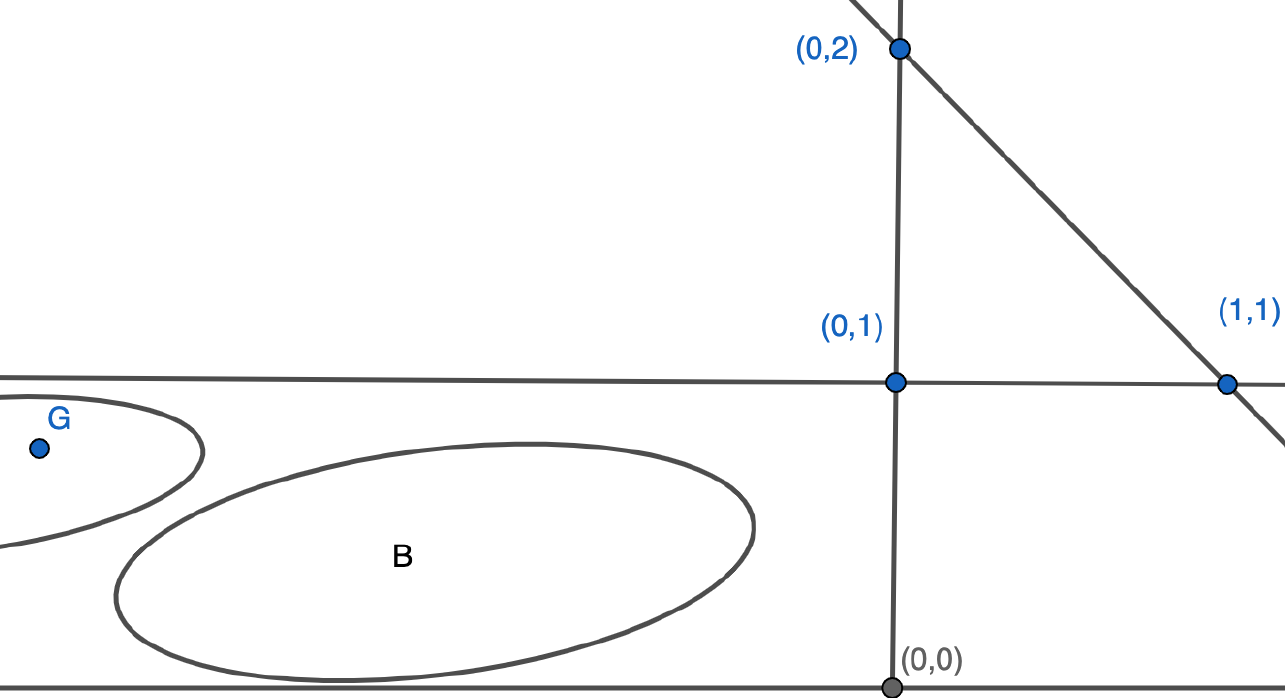}
\end{tabular}
\caption{The vertices of triangle $A$ are mapped to $(1,1),(0,1),(0,2)$ }
\label{fig:masodik}
\end{figure}

\medskip

\begin{figure}[h]
\centering
\includegraphics[width=10cm]{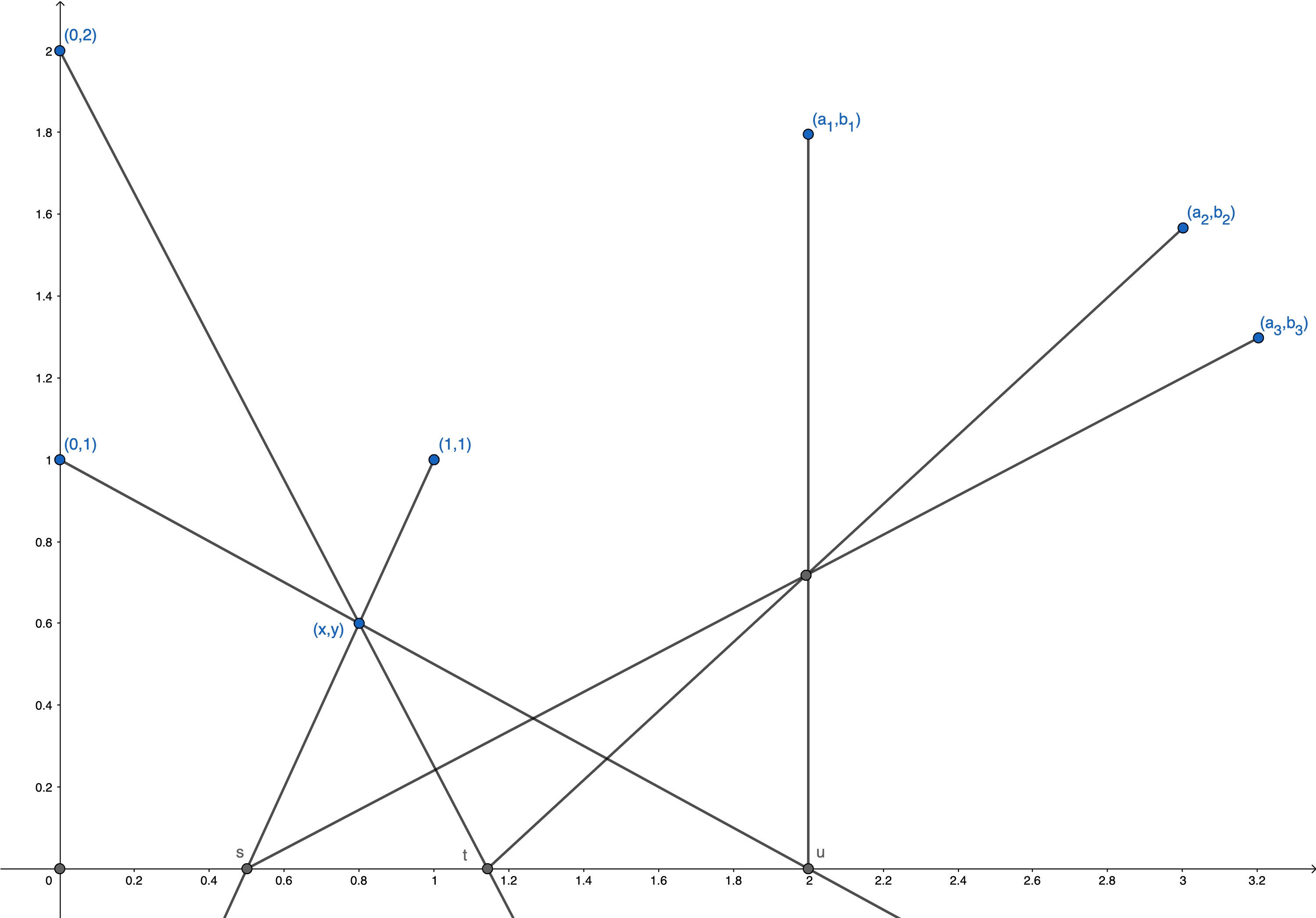}
\caption{With the right selection of points, $(x,y)$ moves along an algebraic curve}
\label{fig:xy_moves}
\end{figure}

\medskip
We denote the image of $(a'_1,b'_1),(a'_2,b'_2),(a'_3,b'_3)$ by $(a_1,b_1),(a_2,b_2),(a_3,b_3),$ 
\[
T_1: (a'_1,b'_1),(a'_2,b'_2),(a'_3,b'_3)\Longrightarrow  (a_1,b_1),(a_2,b_2),(a_3,b_3),
\]
and
\[
T_1: (s'_i,0),(t'_i,0),(u'_i,0)\Longrightarrow  (s_i,0),(t_i,0),(u_i,0) \quad (1\leq i\leq m).
\]

For a point $(x,y),$ where $y\neq 1,2$ the lines through $(x,y)$ and the points $(0,1),(0,2),(1,1)$ will intersect
the $x$ axis in coordinates $u=\frac{x}{1-y},$  $t=\frac{2x}{2-y}$ and $s=\frac{x-y}{1-y}$ (Figure \ref{fig:xy_moves}).
If $(x,y)$ is in $T_1(B)$ as a midpoint of a 3-branch connecting  $(0,1),(0,2),(1,1)$ to $(s_i,t_i,u_i)$ for some $1\leq i\leq m,$ then there is also a 3-branch between  $(a_1,b_1),(a_2,b_2),(a_3,b_3)$ and  $(s_i,t_i,u_i).$ 
Since the three lines of this 3-branch 

\begin{enumerate}
    \item $\ell_1:$ $-b_1X+(a_1-u_i)Y+u_ib_1=0,$
    \item $\ell_2:$ $-b_2X+(a_2-t_i)Y+t_ib_2=0,$
    \item $\ell_3:$ $-b_3X+(a_3-s_i)Y+s_ib_3=0,$
\end{enumerate}
are concurrent, we have the following identity for many $i$

\begin{equation*}
    \begin{vmatrix}
     b_1 & a_{1}-u_i & b_1u_i\\ 
     b_2 & a_{2}-t_i & b_2t_i\\
     b_3 & a_{3}-s_i & b_3s_i 
\end{vmatrix}=0.
\end{equation*}
We want to show that the $(x,y)$ solutions, the zero set of the polynomial below, in equation (\ref{curve}) -- the two-variable polynomial given by the determinant -- is a one-dimensional algebraic set, it is an algebraic curve.
\begin{equation}\label{curve}
\begin{vmatrix}
     (1-y)b_1 & (1-y)a_{1}-x & b_1x\\ 
     (2-y)b_2 & (2-y)a_{2}-2x & 2b_2x\\
     (1-y)b_3 & (1-y)a_{3}-(x-y) & b_3(x-y) 
\end{vmatrix}
=0.
\end{equation}
 Even from the geometric way, we defined the equations, we see that the $y=0$ line is part of the zero set, but we will show that if $y\neq 0,1,2$ then for any given $y$ value, $x$ has at most two possible values. For a fixed $y,$ equation ($\ref{curve}$) is a quadratic equation in $x.$ If it takes three or more different values then all three coefficients of the quadratic equation should be zero. The constant term is

\begin{equation*}
  -b_3y\begin{vmatrix}
     (1-y)b_1 & (1-y)a_{1}\\ 
     (2-y)b_2 & (2-y)a_{2} 
\end{vmatrix}
\end{equation*}
which is zero only if $b_1a_2=a_1b_2.$ The two points, $(a_1,b_1)$ and $(a_2,b_2)$ are on a line through the origin, i.e. there is a $c\neq 0$ such that $(a_1,b_1)=(ca_2,cb_2).$  

The coefficient of the quadratic term is 

\begin{equation*}
  b_1\begin{vmatrix}
     (2-y)b_2 & -2\\
     (1-y)b_3 & -1 
\end{vmatrix}
-2b_2\begin{vmatrix}
     (1-y)b_1 & -1\\
     (1-y)b_3 & -1 
\end{vmatrix}
+b_3\begin{vmatrix}
     (1-y)b_1 & -1\\
     (2-y)b_2 & -2
\end{vmatrix}
=-yb_2(b_1-b_3)
\end{equation*}
which is zero only if $b_1=b_3.$ 

If we just want to see the proof of Theorem \ref{cor:elekes}, then   at this point we are done, since our selection of the two triples of points from the convex point set doesn't allow such accidents. For the proof of Lemma \ref{main3} let's complete the calculations of coefficients. We will see that there are pairs of triples where the zero set of the polynomial in (\ref{curve}) is a two-dimensional set.

The coefficient of the linear term is a bit harder to calculate; it is 

\[
y^2 b_1 b_2 - y a_3 b_1 b_2 (y - 1) - 2 y b_1 b_3 (y - 1) - 2 y b_2 b_3 + y^2 b_2 b_3 + y a_1 b_2 b_3 (y - 1).
\]
Assuming that the other two coefficients are zero, i.e. $b_1=b_3$ and $b_1a_2=a_1b_2,$ this coefficient is

\[
y(y-1)b_2b_3(a_1-a_3-2c+2),
\]
which is zero only if $a_3=a_1-2c+2.$

\medskip
\noindent
If the coordinates of the three points satisfy the equations 
\begin{equation}\label{sim_points}
(a_1,b_1)=(ca_2,cb_2),\quad
b_1=b_3,\quad
a_3=a_1-2c+2
\end{equation}
then, and only then, the zero set of the polynomial in (\ref{curve}) is a two-dimensional set. See for example Figure \ref{fig:two_trip}, where we set the three points as
$(2.5,1.5),(5,3)$ and $(3.5,1.5).$

\begin{figure}[h!]
\centering
\includegraphics[width=14cm]{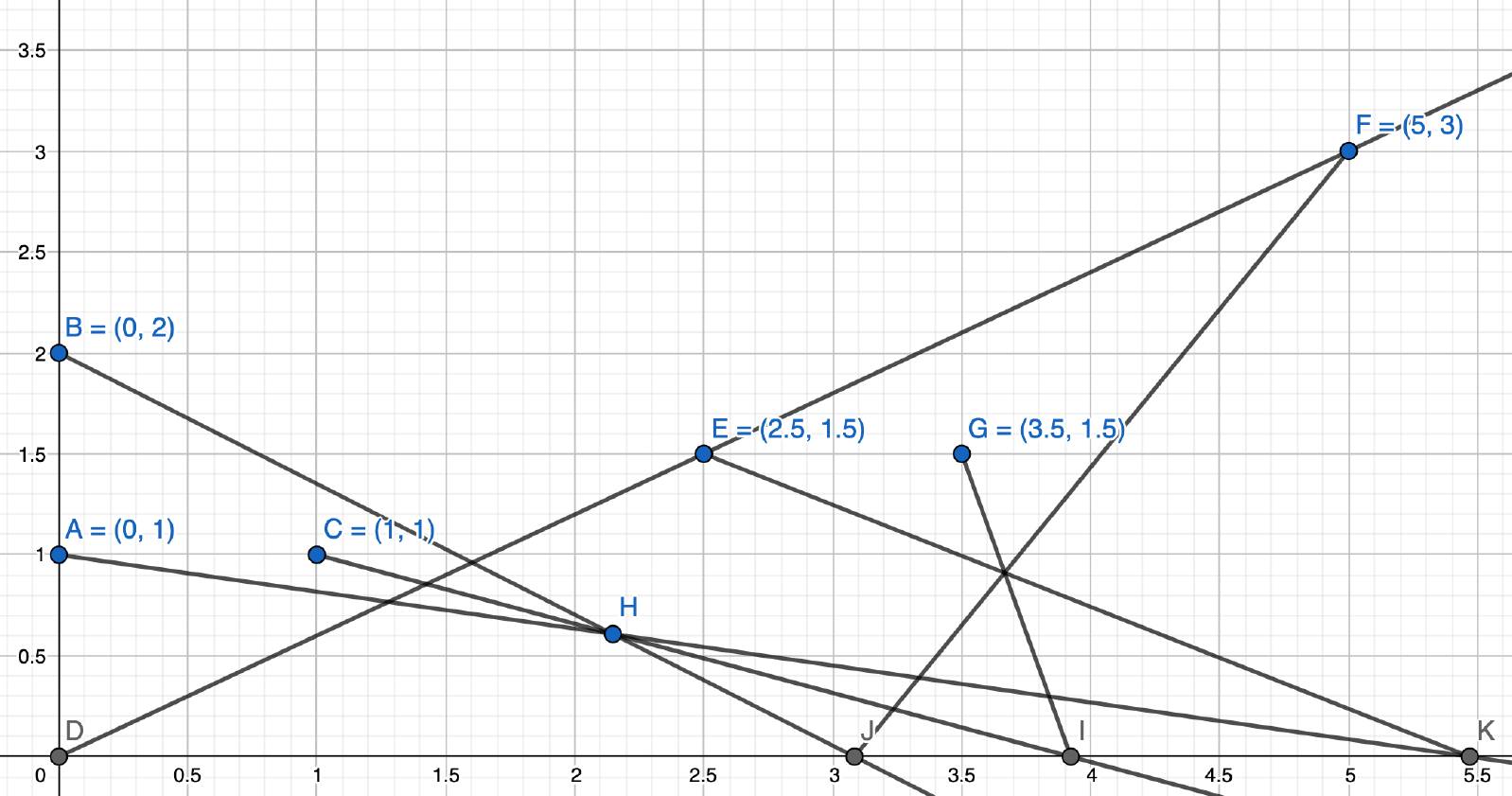}
\caption{With the selected coordinates the three intervals $\overline{EK},\overline{GI}$ and $\overline{FJ}$ are concurrent for any position of $H.$ }
\label{fig:two_trip}
\end{figure}

If all three conditions of (\ref{sim_points}) are satisfied then $(0,1),(0,2),(1,1)$ and $(a_1,b_1),(a_2,b_2),(a_3,b_3)$ are perspective from the $x$-axis. Two common intersection points on the $x$-axis are easy to find from the first two conditions; these are $x=0$ and $x=\infty$. The line of $(0,2)$ and $(1,1)$ intersects at the point $(2,0).$ The equation of the line through $(a_2,b_2),(a_3,b_3)$ is
\[
Y-b_2=\frac{b_2-b_3}{a_2-a_3}(X-a_2).
\]
After substituting the values from (\ref{sim_points}), one can check that this line intersects the $x$ axis at $x=2$ as well. 
For completeness, we include a simple calculation.
Setting $Y=0$ and using the identities in (\ref{sim_points}) we have

\begin{align*}
  b_2(a_3-a_2) & = (b_2-b_3)(X-a_2),\\
  b_2a_3-b_3a_2 & = (b_2-b_3)X,  \\
  b_2a_3-b_1a_2 & = (b_2-b_1)X,  \\
  b_2a_3-b_2a_1 & = b_2(1-c)X,  \\
  a_3-a_1 & = (1-c)X,  \\
  a_1-2c+2 - a_1 & = (1-c)X,  \\
  2 & =X.
\end{align*}

By the assumptions of Lemma \ref{main3}, $A$ does not contain two triples perspective from $C,$ so the equation in (\ref{curve}) defines an algebraic curve. After factoring out the $y=0$ line we were left with a degree two polynomial in two variables 
\begin{align*}
     x^2 b_2 ( b_3 -  b_2) + x \big(y b_1 b_2 -  a_3 b_1 b_2 (y - 1) - 2  b_1 b_3 (y - 1) - 2  b_2 b_3 + y b_2 b_3 +  a_1 b_2 b_3 (y - 1)\big)\\
     +\frac{b_3}{4}(a_1b_2-a_2b_1)\big((2y-3)^2-1\big)= 0,
\end{align*}
which defines a (possibly degenerate) conic, so  at least $c_3\delta^{10}N$ points of $B$ are on a conic.

%This completes the proof of Lemma \ref{main3}, that at least $c_3\delta^{10}N$ points of $B$ are on a conic.
\qed

\subsection{Few directions in a convex point set}\label{directions}
Let's prove Theorem \ref{cor:elekes} first, and then we prove a stronger statement.

\medskip
\begin{proof}[of Theorem \ref{cor:elekes}]:
The convex point set can be separated into two sets $A'$ and $A''$ by a line $\ell$ such that out of the $cn^2$ point pairs at least $\frac{c^2}{2}n^2=c_1n^2$ have one point in $A'$ and the other one in $A''.$
Let's apply a projective transformation which moves the line of infinity to the real plane, to a line $\ell'$ parallel to $\ell,$ which keeps the points of $A$ in convex position on the same halfplane determined by $\ell'$. The set $C$ is defined by the points of directions of the $c_1n^2$ point pairs in $A,$ one point in $A'$ and the other is in $A''.$ From the conditions of the theorem, we have $|C|\leq Dn.$ If, say, $A''$ is closer to $\ell'$ then applying Lemma \ref{main3} with $B:=A''$ and $A:=A',$  shows that at least $c_2n$ points of $A''$ are on a conic, where $c_2>0$ depends on $c$ only. 
\qed
\end{proof}

\medskip
After choosing the separating line $\ell,$ we can select the projective transformation such that either $A'$ or $A''$ are the pointsets closer to $\ell'.$ Because of that, we know that both $A'$ and $A''$ contain many points from conics. However, we don't know whether the two conics are the same or not. The following theorem, one of our main results, addresses this question.

\begin{theorem}\label{one_conic}
For any $D>0$ and  $c>0$ there is a constant, $\delta>0$ and a threshold $n_0=n_0(c,D)$ so the following holds. Let $A$ be a set of $n$ points in the plane in convex position, and $E\subset A\times A$ is a set of point pairs, which determine at most $Dn$ distinct directions. If $|A|\geq n_0$ and $|E|\geq cn^2$, then there is an $H\subset E$ such that $|H|\geq \delta n^2$ and the points of $H$ are on a conic. 
\end{theorem}
\begin{proof}
We follow the previous proof first with the additional step: before applying Lemma \ref{main3}, we remove points from $A''$ (let's assume that $A''$ is closer to $\ell'$) which contribute in few directions with $A'$ in $E$. The set $E'\subset E$ is defined as $E'=E\cap A'\times A''.$ By the previous calculations $|E'|\geq c_1n^2.$ We can consider $E'$ as the edge set of the bipartite graph $G(A',A'').$ Let's remove the vertices from $A''$ which have degree less than $\frac{c_1}{2}n.$ This reduced set is denoted by $A''',$ and $E''=E\cap A'\times A'''.$ We apply Lemma \ref{main3} with parameters $A:=A',$ $B:=A''',$ and $C$ as the points of directions on $\ell'$ (also denoted by $C$).
Now we have a set $A^*\subset A''',$ which is large, $|A^*|\geq c_3n,$ its points are on a conic, $H_1,$ and 
$$|E\cap A'\times A^*|\geq \frac{c_1}{2}n|A^*| =c_4n^2$$ for some constants $c_3,c_4>0$ depending on $c$ and $D$ only.
We keep the points of $A'$ and $A^*,$ and move back $\ell'$ to the infinity and back on the ''other side`` so now $\ell''$ is closer to $A".$
Let's apply Lemma \ref{main3} again, now with $A:=A^*,$ $B:=A'$ and $C$ is the set of direction points on the new $\ell''.$
It follows that there is a $c_5>0$ such that there is a set of points, $A^{**}\subset A'$ which is on a conic, $H_2,$ and $|E\cap A^*\times A^{**}|\geq c_5n^2.$
Now we have two conics and a line containing many collinear triples. 
This answers Elekes' related question, cited after Conjecture \ref{Elekes6} already, but we can go a bit further using a result of Elekes and Szab\'o.

We use the following theorem to show that $H_1$ and $H_2$ are the same conics, which completes our proof.

\begin{theorem}[Elekes-Szab\'o \cite{ESz}]\label{E-Sz}
For every $c > 0$ and positive integer $d$ there exist $\rho = \rho(c, d)$
and $n_0 = n_0(c, d)$ with the following property. Let $\Gamma_1,\Gamma_2,\Gamma_3$ be (not necessarily
distinct) irreducible algebraic curves of degree at most $d$ in the plane $R^2.$ Assume
that $n > n_0$ and
\begin{enumerate}
    \item  no two $\Gamma_i$ are identical straight lines;
    \item  $H_i \subset \Gamma_i$ with $|H_i| \leq n \quad (i = 1, 2, 3)$;
    \item The number of collinear triples $a,b,c$ with $a\in H_1, b\in H_2, c\in H_3$ is at least $cn^2.$
\end{enumerate}

Then $\Gamma_1\cup\Gamma_2\cup\Gamma_3$ is a cubic.
\end{theorem}
In the theorem, it is required that the three curves, in our case $H_1, H_2$ and $\ell'',$ are irreducible. We can suppose that because if $H_1$ or $H_2$ contains a linear component, then by the  result of Elekes we mentioned earlier in the introduction of this section, we had three lines with some $c^*n^2$ collinear triples proving Theorem \ref{one_conic} (Elekes' result is in \cite{El}). 

\qed
\end{proof}

\medskip

\section{Open problems}
Many of the related conjectures of Erd\H{o}s and Elekes  remain open. Here we list a few questions where the tools we used in this paper might be useful.
\begin{enumerate}
    \item In the proofs, we heavily used the fact that the pointsets had a natural ordering. What happens if we consider arbitrary sets in the plane? Do the statements of Theorem \ref{main1} and \ref{main2} still hold?
%    \item If the $cn^2$ collinear triples among $n$ points are 3-term arithmetic progressions, i.e. when in the $(a,b,c)$ triples $b$ is the halving point of the $(a,c)$ segment, one can use standard tools from additive combinatorics to find an arithmetic structure among the points. In particular, there are many collinear points in the set.  It is expected that with the right choice of parameter $\varepsilon,$ if in all the $cn^2$ $(a,b,c)$ triples $b=\lambda a + (1-\lambda)c$ where $1/2-\varepsilon \leq \lambda\leq 1/2+\varepsilon$ then many points are collinear.
    \item Sets with few directions were investigated extensively in affine Galois planes (\cite{ff1,ff2,ff3,ff4,ff5}). Under what conditions can we conclude that Conjecture \ref{Elekes6} holds in planes over finite fields?
    \item Scott's problem on the minimum number of directions was solved in 3-space by Pach, Pinchasi, and Sharir in \cite{PPS}. What structural results can one guarantee if the number of directions determined by two sets in space is very small?
    \item Let us suppose that four mutually avoiding $n$-elements pointsets $A,B,C$ and $D$ have $cn^2$ collinear four-tuples $(a,b,c,d)$ where $a\in A,b\in B,c\in C$ and $d\in D.$ Can we conclude that then at least 5 points are collinear in $A$? (or in any other set). This problem is closely related to a question of Erd\H{o}s \cite{Er}; is it true that if $n$ points contain $cn^2$ collinear 4-tuples, then it should contain five collinear points? For some lower bounds and related results, see \cite{SSs}.
%    \item Given an $n^2$-element pointset such that there is a point in every unit square of an $n\times n$ integer grid. Prove that if the number of collinear triples is $cn^2,$ then many points are collinear, provided $n$ is large enough. 
    \item Are there better bounds for the Brown-Erd\H{o}s-S\'{o}s conjecture for collinear triples than what we had in Theorem \ref{main2}? Maybe -- similar to the case with triples from finite groups -- there is a constant, $c>0,$ such that the following holds. For any $k\in \mathbb{N}$ and $\delta>0$, there is a threshold, $n_0=n_0(k,\delta)$ such that if $n\geq n_0$ and there are at least $\delta n^2$ collinear triples among $n$ points then there is a  $(k,ck^2)$ configuration among the triples.
\end{enumerate}

\section{Acknowledgments}
J.S.'s research was partly supported by an NSERC Discovery grant and OTKA K 119528 grant. The author is thankful to Lior Gishboliner for pointing out the important reference \cite{GS}. The referees' comments significantly improved the presentation of the results and helped me correct a mistake in the previous version.

\end{document}